\documentclass[12pt,twoside,reqno,psamsfonts]{amsart}

\usepackage[OT1]{fontenc}
\usepackage{type1cm}
\usepackage{enumerate}
\usepackage{amssymb}
\usepackage[all]{xy}
\usepackage{mathrsfs}
\usepackage[dvips]{graphicx}
\usepackage{psfrag}          
\usepackage{geometry}        
\usepackage{version}         

\numberwithin{equation}{section}

\theoremstyle{plain}
\newtheorem{theorem}{Theorem}[section]
\newtheorem{corollary}[theorem]{Corollary}

\newtheorem{lemma}[theorem]{Lemma}

\theoremstyle{remark}

\newtheorem{remarks}[theorem]{Remarks}

\newtheorem*{ack}{Acknowledgement}

\theoremstyle{definition}

\newtheorem{question}[theorem]{Question}

\newcommand{\sF}{\mathscr{F}}
\newcommand{\sE}{\mathscr{E}}

\newcommand{\R}{\mathbb{R}}

\newcommand{\N}{\mathbb{N}}

\DeclareMathOperator{\diam}{diam}

\DeclareMathOperator{\spt}{spt}

\DeclareMathOperator{\Geo}{Geo}
\DeclareMathOperator{\GeoOpt}{GeoOpt}

\def\Xint#1{\mathchoice
{\XXint\displaystyle\textstyle{#1}}%
{\XXint\textstyle\scriptstyle{#1}}%
{\XXint\scriptstyle\scriptscriptstyle{#1}}%
{\XXint\scriptscriptstyle\scriptscriptstyle{#1}}%
\!\int}
\def\XXint#1#2#3{{\setbox0=\hbox{$#1{#2#3}{\int}$ }
\vcenter{\hbox{$#2#3$ }}\kern-.6\wd0}}

\def\dashint{\Xint-}

\begin{document}

\title[Local Poincar\'e inequalities on metric spaces]
{Local Poincar\'e inequalities from stable curvature conditions on metric spaces}

\author{Tapio Rajala}
\address{Scuola Normale Superiore\\
Piazza dei Cavalieri 7\\
I-56127 Pisa\\ Italy}
\email{tapio.rajala@sns.it}

\thanks{The author acknowledges the support of the European Project ERC AdG *GeMeThNES* and the Academy of Finland project no. 137528.}
\date{\today}


\begin{abstract}
We prove local Poincar\'e inequalities under various curvature-dimension conditions which are
stable under the measured Gromov-Hausdorff convergence. The first class of spaces we consider
is that of weak $CD(K,N)$ spaces as defined by Lott and Villani.
The second class of spaces
we study consists of spaces where we have a flow satisfying an evolution variational inequality
for either the R\'enyi entropy functional $\sE_N(\rho m) = -\int_X \rho^{1-1/N} dm$ or the Shannon
entropy functional $\sE_\infty(\rho m) = \int_X \rho \log \rho dm$.
We also prove that if the R\'enyi entropy functional is strongly displacement convex in the Wasserstein
space, then at every point of the space we have unique geodesics to almost all points of the space.
\end{abstract}


\maketitle



\section{Introduction}

In the recent years the concept of lower Ricci-curvature bounds has been generalized from the Riemannian setting to the
setting of geodesic metric measure spaces, most prominently in the works of Sturm \cite{S2006I, S2006II} and
Lott and Villani \cite{LV2009}. 
Different authors have studied slightly different definitions of curvature bounds.
What is common to all these definitions is that they strive to fulfill the same set of criteria: to naturally extend
Ricci-curvature bounds of Riemannian manifolds, to be stable under the measured Gromov-Hausdorff convergence, and to 
provide enough structure for meaningful analysis on the metric space.

The first two criteria are well met \cite{S2006I, S2006II, LV2009}, whereas the third one has so far been met only partially.
From a classical result by Buser \cite{B1982} we know that a Riemannian manifold with nonnegative Ricci-curvature supports a
Poincar\'e inequality. Moreover, in the case of measured Gromov-Hausdorff limits of Riemannian manifolds with Ricci-curvature
bounded below we know that a local Poincar\'e inequality always holds \cite{CC2000}.
In order to prove a local Poincar\'e inequality for a larger class of metric spaces with Ricci-curvature lower-bounds the stability under
measured Gromov-Hausdorff convergence had to be sacrificed in \cite{LV2007} by assuming the spaces to be nonbranching.
The purpose of this paper is to complete this part of the theory for some versions of curvature-dimension conditions by proving
local Poincar\'e inequalities without the unnatural extra assumption of nonbranching. The local Poincar\'e inequality
together with a doubling measure constitute the very foundation of analysis on metric spaces
which was pioneered by Heinonen and Koskela \cite{HK1998} and Cheeger \cite{C1999}. For an introduction to analysis on
metric spaces we refer to the book by Heinonen \cite{H2001}. Poincar\'e inequalities have been proven in many classes of
metric spaces, for example in locally linearly contractible Ahlfors-regular metric spaces \cite{S1996}.

A metric measure space $(X,d,m)$ admits a weak local $(q,p)$-Poincar\'e inequality with $1 \le p \le q < \infty$ if there
exist constants $\lambda \ge 1$ and $ 0 < C < \infty$ such that for any continuous function $u$ defined on $X$, any point 
$x \in X$ and radius $r>0$ such that $m(B(x,r)) > 0$ and any upper gradient $g$ of $u$ we have
\begin{equation}\label{eq:Poincaredefinition}
  \left(\dashint_{B(x,r)}|u - \langle u\rangle_{B(x,r)}|^qdm\right)^{1/q} \le C r \left(\dashint_{B(x,\lambda r)}g^pdm\right)^{1/p},
\end{equation}
where the barred integral denotes the average integral and $\langle u\rangle_{B(x,r)}$ denotes the average of $u$ in the
ball $B(x,r)$. Recall that, as introduced in \cite{HK1998}, a Borel function $g \colon X \to [0,\infty]$ is an upper gradient
of $u$ if for any constant speed curve $\gamma \colon [0,1] \to X$ with length $l(\gamma) < \infty$ we have
\[
 |u(\gamma(0)) - u(\gamma(1))| \le l(\gamma) \int_0^1g(\gamma(t))dt.
\]

The most studied Poincar\'e inequalities are $(1,p)$-Poincar\'e inequalities. From H\"older's inequality it follows immediately
that a weak local $(1,p)$-Poincar\'e inequality implies a weak local $(1,p')$-Poincar\'e inequality for every $p' > p$.
In this paper we consider weak local $(1,1)$-Poincar\'e inequalities which we simply call weak local Poincar\'e inequalities.
The word \emph{weak} in the weak local Poincar\'e inequality refers to the fact that we allow the ball on the right-hand side of \eqref{eq:Poincaredefinition} to be
larger than the one on the left. If the balls on both sides of the inequality can be taken to be the same, meaning that we can
take $\lambda = 1$, the inequality is called a strong local Poincar\'e inequality. In a doubling geodesic metric space the weak
local Poincar\'e inequality implies the strong one, with possibly a different constant $C$, see \cite{HK1995} and also \cite{HK2000}.
We say that a measure $m$ is doubling (with a constant $1\le D<\infty$) if for all $x \in X$ and $0 < r < \diam(X)$ we have
\[
 m(B(x,2r)) \le D m(B(x,r)).
\]

Lott and Villani gave in \cite{LV2009} a definition for nonnegative $N$-Ricci curvature with $N \in [1,\infty)$  and a definition
for $\infty$-Ricci curvature being bounded below by $K \in \R$. In their definition they required weak displacement convexity
for a collection $\mathcal{DC}_N$ of suitable convex functionals. (For the precise definitions, see Section \ref{sec:definitions}.)
More specifically they required that between any two probability measures that have bounded Wasserstein distance between them
there is at least one geodesic in the Wasserstein space of probability measures along which all the 
functionals in $\mathcal{DC}_N$ satisfy a convexity inequality. At the same time when Lott and Villani defined their curvature-dimension bounds
another definition was given by Sturm \cite{S2006I, S2006II}. He defined spaces where $N$-Ricci curvature is bounded from below by
a constant $K \in \R$. In Sturm's definition the displacement convexity is also required along only one geodesic but for only one 
critical functional. Later Lott and Villani combined the two different definitions by considering what we call here $CD(K,N)$ 
spaces with $K \in \R$ referring to the lower bound on curvature and $N \in [1,\infty]$ to the upper bound on dimension.
In this definition weak displacement convexity is again required for the whole collection $\mathcal{DC}_N$.
The assumptions in the definition  by Sturm are a priori weaker. However, in nonbranching metric spaces the two notions agree, 
see for example \cite{V2009}. It should be emphasized that our point is to avoid making the nonbranching assumption. Therefore 
the results we get for $CD(K,N)$ spaces do not necessarily hold under the definitions of Sturm.

Let us present the Poincar\'e inequalities we are able to obtain. The first result is for the large class of $CD(K,\infty)$ spaces.
\begin{theorem}\label{thm:CD1}
 Suppose that $(X,d,m)$ is a $CD(K,\infty)$ space with $K \le 0$. Then the weak local Poincar\'e inequality
 \[
  \int_{B(x,r)}|u - \langle u\rangle_{B(x,r)}|dm \le 4 r e^{|K|r^2}\int_{B(x,2 r)}gdm
 \] 
 holds for any continuous function $u$ defined on $X$, any upper gradient $g$ of $u$ and for each point 
 $x \in X$ and radius $r>0$.
\end{theorem}

We do not have average integrals in Theorem \ref{thm:CD1} like we had in \eqref{eq:Poincaredefinition}.
If we were to write the average integrals here we would have to multiply the right-hand side of the inequality
by a factor $m(B(x,2r))/m(B(x,r))$. This factor could well be unbounded as $r \downarrow 0$ in an infinite dimensional space. 
Although for $CD(K,\infty)$ spaces the assumption for $m$ to be doubling is not very natural
we always have doubling for $CD(K,N)$ spaces when $N<\infty$: by the Bishop-Gromov inequality any $CD(K,N)$ space $(X,d,m)$
is doubling with a constant $2^N$, see for example \cite{V2009}. Therefore for $CD(K,N)$ spaces we can
write our local Poincar\'e inequalities in a more standard form.

\begin{theorem}\label{thm:CD2}
 Any $CD(K,N)$ space with $K \le 0$ supports the weak local Poincar\'e inequality
 \[
  \dashint_{B(x,r)}|u - \langle u\rangle_{B(x,r)}|dm \le 2^{N+2} r e^{\sqrt{(N-1)|K|}2r} \dashint_{B(x,2 r)}gdm.
 \]
 In particular, any $CD(0,N)$ space supports the uniform weak local Poincar\'e inequality
 \[
  \dashint_{B(x,r)}|u - \langle u\rangle_{B(x,r)}|dm \le 2^{N+2} r \dashint_{B(x,2 r)}gdm.
 \]
\end{theorem}

Notice that the constant $2^{N+2}$ here is better than the constant $2^{2N+1}$ obtained in \cite{LV2007} for nonbranching
$CD(0,N)$ spaces. So far we have considered $CD(K,N)$ spaces where the weak displacement convexity is required for a class
of functionals. We are unable to prove the previous two theorems from the curvature-dimension conditions of Sturm where the
weak displacement convexity is required for only one functional.

The nonnegativity of $N$-Ricci curvature can also be generalized to metric measure spaces using gradient flows.
This is done by requiring the existence of a flow satisfying the so called evolution variational inequality (E.V.I.)
\[
 \limsup_{h \downarrow 0}\frac{1}{2h}\left(W_2^2(\nu,\mu_{t+h})-W_2^2(\nu,\mu_{t})\right) \le \sE_N(\nu) - \sE_N(\mu_t)
\]
for all $t \ge 0$ and for any probability measure $\nu$ which is absolutely continuous with respect to the measure $m$. Here $W_2$
is the Wasserstein distance between measures which will be defined in Section \ref{sec:definitions}. 
For example in compact Alexandrov spaces with curvature bounded below an E.V.I. condition is satisfied \cite{O2007, GO2010}.
The functional $\sE_N$ in the E.V.I. is the critical entropy functional which depends on the dimension bound $N$. 
For finite $N$ it is the R\'enyi entropy functional $\sE_N$ which corresponds to the $CD(0,N)$ spaces. For $N = \infty$
the critical functional is the Shannon entropy functional $\sE_\infty$ corresponding to $CD(0,\infty)$ spaces.

The E.V.I. condition is also stable under the measured Gromov-Hausdorff convergence (see for instance \cite{S2007} for this type 
of statement). It is perhaps easier to relate the E.V.I. condition to the 
previous two curvature-dimension conditions after noticing that it implies the strong displacement convexity of the entropy functional 
corresponding to the E.V.I., see \cite{DS2008}. This means that the functional is not only displacement convex along one geodesic
between any two measures, as was required in the definitions by Sturm, and Lott and Villani, but it is in fact displacement convex
along any geodesic between any two measures. This strong displacement convexity is not by itself stable under the measured 
Gromov-Hausdorff convergence and therefore it is a poor notion of curvature-dimension bound.
Clearly the curvature-dimension condition of Sturm is also weaker than strong displacement convexity and thus weaker than the E.V.I. condition.
What is not so clear is how the E.V.I. condition relates to the $CD(0,N)$ spaces.
The following diagram gathers the implications which we are aware of between the local Poincar\'e inequality and the different 
curvature-dimension conditions that are mentioned in this paper. 
Notice that in this paper we consider the E.V.I. condition and strong displacement convexity of $\sE_N$ only for the curvature
lower-bound $K = 0$. We are not aware of a generalization of the E.V.I. condition for all $K \in \R$ and $N \in [1,\infty]$.
The curvature-dimension condition $CD(K,N)$ however works for all $K$ and $N$.
\[
\xymatrix {
   & *+[F-,]{\txt{\textsc{Sturm}: \\weak displacement\\ convexity for $\sE_N$ \\ \emph{(stable condition)}}}  
   & *+[F-,] {\txt{\textsc{Lott-Villani}: \\weak displacement\\ convexity for all $\sF \in \mathcal{DC}_N$ \\ \emph{(stable condition)}}} \ar @{=>} [l] \ar @{=>} [d]\\
 *+[F-,]{\txt{E.V.I. for $\sE_N$\\ \emph{(stable condition)}}} \ar @{=>} [r]
   & *+[F-,]{\txt{strong displacement\\ convexity for $\sE_N$\\ \emph{(unstable condition)}}} \ar @{=>} [u] \ar @{=>} [r]
   & *+[F]{\txt{local Poincar\'e\\ inequality}}
}
\]

As indicated by the previous diagram we have a local Poincar\'e inequality (analogous to Theorem \ref{thm:CD1}) for spaces where the 
entropy functionals are strongly displacement convex, and therefore also for the spaces satisfying the corresponding E.V.I. condition.

\begin{theorem}\label{thm:EVI1}
 Suppose that $\sE_\infty$ (or $\sE_N$) is strongly displacement convex in a geodesic metric measure space $(X,d,m)$.
 Then the weak local Poincar\'e inequality
 \[
  \int_{B(x,r)}|u - \langle u\rangle_{B(x,r)}|dm \le 4 r \int_{B(x,2 r)}gdm
 \] 
 is satisfied.
\end{theorem}

When we combine the doubling property of the measure $m$ in the $N<\infty$ case with Theorem \ref{thm:EVI1} we get again
a more standard version of local Poincar\'e inequality.

\begin{corollary}\label{cor:EVI2}
 Suppose that $\sE_N$ is strongly displacement convex in a geodesic metric measure space $(X,d,m)$.
 Then the space $X$ admits the weak local Poincar\'e inequality
 \[
  \dashint_{B(x,r)}|u - \langle u\rangle_{B(x,r)}|dm \le 2^{N+2} r \dashint_{B(x,2 r)}gdm.
 \]
\end{corollary}

To once more emphasize that not all of the desired connections between local Poincar\'e inequalities and
curvature-dimension bounds are known, we state the missing ones corresponding to $K=0$ here as a question.

\begin{question}\label{q:1}
 Does the conclusion of Theorem \ref{thm:EVI1} or Corollary \ref{cor:EVI2} remain valid if we require only
 weak displacement convexity instead of strong displacement convexity?
\end{question}

\subsection{Almost uniqueness of geodesics}

Let us now slightly change the perspective to the question of the nonbranching assumption.
In the earlier proofs of local Poincar\'e inequalities nonbranching was assumed in order to guarantee
almost uniqueness of geodesics. It is natural to ask to what extent in general we have uniqueness of 
geodesics under the curvature-dimension conditions. We prove in the last section of this paper
the following theorem, Theorem \ref{thm:uniqueness}, which says that the strong displacement convexity of $\sE_N$, and therefore
also the corresponding E.V.I. condition, implies that from every
point in the underlying space there is a unique geodesic to $m$-almost every point in the space. In connection with the local
Poincar\'e inequalities above it means that we could also prove Corollary \ref{cor:EVI2} with a larger 
constant following more directly the proof of Lott and Villani \cite{LV2009}.

\begin{theorem}\label{thm:uniqueness}
 Suppose that $\sE_N$ is strongly displacement convex in a geodesic metric measure space $(X,d,m)$.
 Then for every $x \in X$
 \[
  m(\{y \in X : \text{there exist two distinct geodesics between }x \text{ and }y\}) = 0.
 \] 
\end{theorem}

We do not know if the conclusion of Theorem \ref{thm:uniqueness} holds also for the Shannon entropy functional $\sE_\infty$.
The proof of Theorem \ref{thm:uniqueness} does not seem to work in this situation.
Notice that in order to have a reasonable formulation of this result for $\sE_\infty$ the conclusion has to be weakened so that it does
not involve any singular parts of the measures. The question could be stated for example in the following way.

\begin{question}
 Suppose that $\sE_\infty$ is strongly displacement convex in a geodesic metric measure space $(X,d,m)$.
 Do we then have for every $\mu, \nu \in \mathcal{P}^{ac}(X,m)$ for which we have $W_2(\mu,\nu)<\infty$,
 and for every $\pi \in \GeoOpt(\mu,\nu)$ uniqueness of geodesics in the sense that
 \[
  \pi(\{\gamma \in \Geo(X) : \text{ there exist two distinct geodesics between }\gamma(0) \text{ and }\gamma(1)\}) = 0?
 \]
\end{question}

Another way to formulate the uniqueness is to rely on the measure $m$ without using optimal plans.

\begin{question}
 Suppose that $\sE_\infty$ is strongly displacement convex in a geodesic metric measure space $(X,d,m)$.
 Do we then have 
 \[
  m\otimes m(\{(x,y) \in X\times X : \text{ there exist two distinct geodesics between }x \text{ and }y\}) = 0?
 \]
\end{question}

In the next section the relevant preliminaries will be given. Section \ref{sec:Poincare} contains
the proofs of all the local Poincar\'e inequalities presented here. In the final section, Section \ref{sec:uniqueness},
we prove Theorem \ref{thm:uniqueness} on the almost uniqueness of geodesics under the strong displacement convexity of $\sE_N$.

\begin{ack}
 Many thanks are due to Luigi Ambrosio for his lectures on optimal mass transportation in the spring of 2011 where I learned
 this topic, for the many discussions I have had with him around the subject and, in particular, for pointing out to me
 the question of uniqueness of geodesics under the strong displacement convexity. Thanks are also due to the anonymous referee
 for suggesting many improvements to the paper.
\end{ack}

\section{Preliminaries}\label{sec:definitions}

For the most parts we will follow the notation used in \cite{LV2009}. The metric measure spaces $(X,d,m)$ we consider are 
always geodesic and complete. Furthermore, the support of the measure $m$ can always be assumed to be the whole space $X$.
We denote the support of a measure $\mu$ by $\spt \mu$.
Let us denote by $\mathcal{P}(X)$ the Borel probability measures on $X$ and by $\mathcal{P}^{ac}(X,m) \subset \mathcal{P}(X)$
the probability measures in $X$ that are absolutely continuous with respect to the measure $m$.

Recall that any geodesic in a metric space $(X,d)$ can be reparametrized to be a continuous mapping $\gamma \colon [0,1] \to X$ with
\[
 d(\gamma(t),\gamma(s)) = |t-s|d(\gamma(0),\gamma(1)) \qquad \text{for all } 0 \le t \le s \le 1.
\]
We denote the space of all the geodesics of the space $X$ with such parametrization by $\Geo(X)$.
An important concept for this paper is (non)branching in geodesic metric spaces. By branching of geodesics
we mean that there are some distinct geodesics starting from the same point which follow the same path for some
initial time interval and then become disjoint.

\subsection{Measured Gromov-Hausdorff convergence}

Let us recall what we mean by measured Gromov-Hausdorff 
convergence even though we will not need it in the proofs. A sequence of compact metric measure spaces
$\{(X_i,d_i,m_i)\}_{i=1}^\infty$ is said to converge to a compact metric measure space $(X,d,m)$ in the measured
Gromov-Hausdorff sense, if there exists a sequence of Borel maps $f_i \colon X_i \to X$ and a sequence of positive numbers 
$\epsilon_i \downarrow 0$ so that
\begin{enumerate}
 \item For every $x_i,y_i \in X_i$ we have $|d(f_i(x_i),f_i(y_i))-d_i(x_i,y_i)| < \epsilon_i$.
 \item For every $x \in X$ and $i \in \N$ there exists $x_i \in X_i$ such that $d(f_i(x_i),x) \le \epsilon_i$.
 \item We have $(f_i)_\# m_i \to m$ as $i \to \infty$ in the weak-* topology of $\mathcal{P}(X)$.
\end{enumerate}

\subsection{Wasserstein space $(\mathcal{P}(X), W_2)$}

The convexity of the functionals will be considered along the geodesics in the Wasserstein space $(\mathcal{P}(X), W_2)$.
The distance between two probability measures $\mu, \nu \in \mathcal{P}(X)$ in this space is given by
\[
 W_2(\mu, \nu) = \left(\inf\left\{\int_{X\times X} d(x,y)^2d\sigma(x,y)\right\}  \right)^{1/2},
\]
where the infimum is taken over all $\sigma \in \mathcal{P}(X \times X)$ with $\mu$ as its first marginal and $\nu$ as the second,
i.e. $\mu(A) = \sigma(A \times X)$ and $\nu(A) = \sigma(X \times A)$ for all Borel subsets $A$ of the space $X$. Notice that in the
case where the distance $d$ is not bounded the function $W_2$ is strictly speaking not a distance as the
above infimum can also take an infinite value. 

Now, if it happens that there is a geodesic $\Gamma \in \Geo(\mathcal{P}(X))$ between two measures $\mu, \nu \in \mathcal{P}(X)$ in the space $(\mathcal{P}(X), W_2)$
it must be that this geodesic can be realized as a measure $\pi \in \mathcal{P}(\Geo(X))$ so that $\Gamma(t) = (e_t)_\#\pi$, where
$e_t(\gamma) = \gamma(t)$ for any geodesic $\gamma$ and $t \in [0,1]$ and $f_\#\mu$ denotes the push-forward of the measure $\mu$
under $f$, see for example \cite[Corollary 7.22]{V2009}. This realization is convenient for us when we want to
translate information from the geodesics on $\mathcal{P}(X)$ to the geodesics on $X$.
The space consisting of all measures $\pi\in \mathcal{P}(\Geo(X))$
for which the mapping $t \mapsto (e_t)_\#\pi$ is a geodesic in $\mathcal{P}(X)$ from $\mu = (e_0)_\#\pi$ to $\nu = (e_1)_\#\pi$ 
is denoted by $\GeoOpt(\mu, \nu)$. We refer to \cite{V2009} for a detailed account on the theory of optimal transportation
which includes all the basic information on the Wasserstein space.

\subsection{Convex functionals}

All the curvature-dimension conditions we consider in this paper are defined using integral functionals. For $N \in [1,\infty)$
the functionals are build using functions in the displacement convexity class $\mathcal{DC}_N$ consisting of all the continuous 
convex functions $F \colon [0,\infty) \to \R$ for which we have $F(0) = 0$ and for which the function
\[
 \lambda \mapsto \lambda^NF(\lambda^{-N})
\]
is convex on the interval $(0,\infty)$. For $N = \infty$ we use the class $\mathcal{DC}_\infty$ of continuous convex functions
$F \colon [0,\infty) \to \R$ for which we have $F(0) = 0$ and for which the function
\[
 \lambda \mapsto e^\lambda F(e^{-\lambda})
\]
is convex on the interval $(-\infty,\infty)$. Notice that $\mathcal{DC}_N \subset \mathcal{DC}_{N'}$ for
all $1 \le N' \le N \le \infty$. This means that the requirement of satisfying some condition for all $F \in \mathcal{DC}_N$
becomes more and more demanding as we decrease $N$.

Most of the functionals $\sF \colon \mathcal{P}(X) \to \R$ which we consider are of the form
\begin{equation}\label{eq:functionaldef}
 \sF(\mu) = \int_X F(\rho)dm + F'(\infty)\mu^\perp(X),
\end{equation}
where the measure $\mu$ is decomposed to the absolutely continuous part $\rho m$ and the singular part $\mu^\perp$ with respect to $m$.
The derivative at infinity is defined as
\[
F'(\infty) = \lim_{r \to \infty}\frac{F(r)}{r} 
\]
which is guaranteed to exist for all $F \in \mathcal{DC}_1$ because of the convexity assumption.

Particularly relevant integral functionals for the curvature-dimension conditions 
are defined via \eqref{eq:functionaldef} using as $F$ the functions
\[
 F_N(r) = - r^{1-\frac1N} \quad \text{and} \quad F_\infty(r) = r \log r,
\]
where $N \in [1,\infty)$. Notice that $F_N \in \mathcal{DC}_N$ and $F_\infty \in \mathcal{DC}_\infty$ and that for these functions we have
\[
 F_N'(\infty) = 0 \quad \text{and} \quad F_\infty'(\infty) = \infty
\]
meaning that the R\'enyi entropy functional $\sE_N$ built from the function $F_N$ does not see the singular part
of the measure $\mu$ whereas the Shannon entropy functional $\sE_\infty$, corresponding respectively to $F_\infty$, in the presence of
any singular part has value $\infty$.

\subsection{Displacement convexity}

The notion of displacement convexity was first used by McCann in \cite{McC1997}.
The functional $\sF \colon \mathcal{P}(X) \to \R\cup\{-\infty,\infty\}$ is called 
\emph{strongly displacement convex} if for any $\Gamma \in \Geo(\mathcal{P}(X))$ we have
\begin{equation}\label{eq:lambdaconvex}
  \sF(\Gamma(s)) \le (1-s)\sF(\Gamma(0)) + s\sF(\Gamma(1))
\end{equation}
for all $s \in [0,1]$. The functional $\sF$ is called \emph{weakly displacement convex} if for any
two measures $\mu, \nu$ there exists $\Gamma \in \Geo(\mathcal{P}(X))$ so that $\Gamma(0) = \mu$, $\Gamma(1) = \nu$ and 
$\Gamma$ satisfies the inequality \eqref{eq:lambdaconvex}. In general only the implication
\begin{center}
 strong displacement convexity $\Rightarrow$ weak displacement convexity
\end{center}
holds. In the particular case of Riemannian manifolds the converse is also true \cite{FV2007}. The converse
is also true if the space $X$ is nonbranching, see \cite[Theorem 30.32]{V2009}.

\subsection{Evolution variational inequalities}

As was mentioned in the introduction one good $N$-dimensional condition, with $N \in [1,\infty]$, for nonnegative curvature
is the requirement of the existence of a flow satisfying the so called \emph{evolution variational inequality} (E.V.I.)
\[
 \limsup_{h \downarrow 0}\frac{1}{2h}\left(W_2^2(\nu,\mu_{t+h})-W_2^2(\nu,\mu_{t})\right) \le \sE_N(\nu) - \sE_N(\mu_t)
\]
for all $t \ge 0$ and $\nu \in \mathcal{P}^{ac}(X,m)$. See \cite{AGS2008} for a comprehensive introduction to E.V.I. and
gradient flows. The E.V.I. condition is stable under measured Gromov-Hausdorff convergence. This can be proven in a similar way
as \cite[Theorem 7.12]{AG2011}. The existence of a flow satisfying the E.V.I. implies that the entropy $\sE_N$ is strongly
displacement convex. This was shown in a very general setting for functionals in metric spaces in \cite{DS2008}.

\subsection{Curvature-dimension conditions $CD(K,N)$}

In order to define the $CD(K,N)$ spaces we need some more notation. Given $K \in \R$ and $N \in (1,\infty]$, we define
\[
 \beta_t(x,y) = \begin{cases}
                 e^{\frac{1}{6}K(1-t^2)d(x,y)^2}  & \text{if }N = \infty,\\
                 \infty  & \text{if }N < \infty, K> 0 \text{ and }\alpha >\pi,\\
                 \left(\frac{\sin(t\alpha)}{t\sin \alpha}\right)^{N-1}  & \text{if }N < \infty, K> 0 \text{ and }\alpha \in [0,\pi],\\
                 1  & \text{if }N < \infty \text{ and } K = 0,\\
                 \left(\frac{\sinh(t\alpha)}{t\sinh \alpha}\right)^{N-1}  & \text{if }N < \infty \text{ and } K < 0,
                \end{cases}
\]
where
\[
 \alpha = \sqrt{\frac{|K|}{N-1}}d(x,y).
\]
For $N = 1$ we define
\[
 \beta_t(x,y) = \begin{cases}
                 \infty  & \text{if } K> 0, \\
                 1 & \text{if }K \le 0.
                \end{cases}
\]

We say that $(X,d,m)$ is a $CD(K,N)$ space, with the interpretation that it has $N$-Ricci curvature bounded below by $K$,
if for any two measures $\mu_0, \mu_1 \in \mathcal{P}(X)$ with $W_2(\mu_0,\mu_1)<\infty$ there
exists $\pi \in \GeoOpt(\mu_0,\mu_1)$ so that along the Wasserstein geodesic $\mu_t = (e_t)_\#\pi$ for every $F \in \mathcal{DC_N}$
and for every $t \in [0,1]$ we have
\begin{align}
 \sF(\mu_t) \le &~ (1-t)\iint_{X\times X}\beta_{1-t}(x_0,x_1)F\left(\frac{\rho_0(x_0)}{\beta_{1-t}(x_0,x_1)}\right)d\sigma(x_1|x_0)dm(x_0) \nonumber\\
                &~ + t\iint_{X\times X}\beta_{t}(x_0,x_1)F\left(\frac{\rho_1(x_1)}{\beta_{t}(x_0,x_1)}\right)d\sigma(x_0|x_1)dm(x_1) \nonumber\\
                &~ + F'(\infty)\left((1-t)\mu_{0}^\perp + t\mu_{1}^\perp \right), \label{eq:CDdef}
\end{align}
where we have written $\mu_0 = \rho_0m + \mu_0^\perp$ and $\mu_1 = \rho_1m + \mu_1^\perp$ to their absolutely continuous and singular
parts with respect to the measure $m$, and where $d\sigma(x_0|x_1)$ denotes the disintegrated measure of $\sigma = (e_0,e_1)_\#\pi$ with
respect to $\mu_1$ and $d\sigma(x_1|x_0)$ with respect to $\mu_0$. We will not be precise about this disintegration
as it will not be needed here. This is so because in the proofs we will use only measures $\mu_0, \mu_1 \in \mathcal{P}^{ac}(X,m)$ 
and for such measures we know by \cite[Lemma 29.6]{V2009} that the inequality \eqref{eq:CDdef} can be written as
\begin{align}
 \sF(\mu_t) \le &~ (1-t)\iint_{X\times X}\frac{\beta_{1-t}(x_0,x_1)}{\rho_0(x_0)}F\left(\frac{\rho_0(x_0)}{\beta_{1-t}(x_0,x_1)}\right)d\sigma(x_0,x_1) \nonumber\\
                &~ + t\iint_{X\times X}\frac{\beta_{t}(x_0,x_1)}{\rho_1(x_1)}F\left(\frac{\rho_1(x_1)}{\beta_{t}(x_0,x_1)}\right)d\sigma(x_0,x_1). \label{eq:CDdefabs}
\end{align}

Important thing to notice from the definition of $CD(K,N)$ spaces is that any $CD(K,N)$ space is also a $CD(K',N)$ space for every
$K' \le K$ and a $CD(K,N')$ space for any $N' \ge N$.

\section{Proof for the Poincar\'e inequalities}\label{sec:Poincare}

Before proving the Poincar\'e inequalities mentioned in the introduction let us point out the main differences between this proof 
and the proof in \cite{LV2007} for the case of nonbranching spaces. We know that in a nonbranching $CD(K,N)$ metric space 
with $N \in [1,\infty)$ there exists a unique geodesic from every point $x$ to $m$-almost every point $y$,
see \cite[Theorem 30.17]{V2009}. Therefore for every point $x$ and any ball $B$ we have a unique geodesic $\pi$ between $\delta_x$
and $\frac1{m(B)}m|_B$. Then for example in the $CD(0,N)$ case we can use the displacement convexity of $\sE_N$
along this geodesic to obtain for each $t\in [0,1]$ a bound on the density $\rho_t$ of $(e_t)_\#\pi$ with respect to $m$ of the form
\[
 \rho_t(y) \le \frac1{t^Nm(B)}
\]
at $m$-almost every point $y$. Combining such estimates for all points $x \in B$ to obtain a so called
dynamical democratic transference plan and using the arguments of Lott and Villani, most notably the symmetry of the estimates
in time, we deduce a Poincar\'e inequality
\[
 \int_{B(x,r)}|u - \langle u\rangle_{B(x,r)}|dm \le 2^{N+1} r \int_{B(x,2 r)}gdm
\]
for the case $CD(0,N)$, with $N \in [1,\infty)$. This approach has also been used by other authors, see \cite{vR2008} for a
similar proof of the local Poincar\'e inequality in nonbranching spaces using a related measure contraction property
 \cite{O2007, S2006II}.

The key observation for proving better local Poincar\'e inequalities is that it is sufficient to split the ball $B$ into two equal
sized parts $B^+$ and $B^-$ using the median of the function $u$ in the ball $B$ and then to find a geodesic $\pi$ between
$\frac1{m(B^+)}m|_B^+$ and $\frac1{m(B^-)}m|_B^-$ along which
we have a good estimate from above to the density $\rho_t$. In the case $K=0$ the bound we obtain is the best possible one: 
for all $t \in [0,1]$
\[
 \rho_t(y) \le \frac2{m(B)}
\]
at $m$-almost every point $y$. We have two rather standard ways of obtaining this bound. In the strong displacement convexity 
case we can prove it by considering restrictions of the optimal plan $\pi$. This argument has been used for example in
\cite[Chapter 19]{V2009}. For the $CD(K,N)$ spaces the bound can be proven by 
taking a sequence of more and more convex functionals in $\mathcal{CD}_N$, just like in the proof of \cite[Theorem 30.20]{V2009}.

We will write the density bounds as their own lemmas and then prove Theorem \ref{thm:CD1} from the corresponding density bound.
The rest of the weak local Poincar\'e inequalities follow analogously.
Let us start the proofs with the density bounds in the $CD(K,N)$ spaces.

\begin{lemma}\label{lma:CD}
 Suppose that $(X,d,m)$ is a $CD(K,N)$ space with $K \le 0$ and $N \in [1,\infty]$.
 Let $\mu, \nu \in \mathcal{P}^{ac}(X,m)$ be measures with densities bounded from above by a constant $c$
 and so that $W_2(\mu,\nu)<\infty$.
 Suppose also that $D = \diam(\spt \mu \cup \spt \nu) < \infty$. Then there exists $\pi \in \GeoOpt(\mu,\nu)$ so that
 for every $t \in [0,1]$ we have
 \[
   \rho_t(x) \le \begin{cases}      
               e^{\frac{1}{6}|K|D^2}c & \text{if }N = \infty,\\
               e^{\sqrt{(N-1)|K|}D}c & \text{if }N < \infty
             \end{cases}
 \]
 at $m$-almost every $x \in X$, where $\rho_t$ is the density of $(e_t)_\#\pi$ with respect to $m$.
\end{lemma}
\begin{proof}
 Take the geodesic $\pi \in \GeoOpt(\mu,\nu)$ along which every $F \in \mathcal{DC}_N$ satisfies \eqref{eq:CDdef}.
 Take $p\ge 1$ and define $F(r) = r^p$. Because $F \in \mathcal{DC}_N$ for every $N \in [1,\infty]$ and the measures $\mu$ and $\nu$
 have no singular part with respect to $m$, we get by \eqref{eq:CDdefabs}
\begin{align*}
 ||\rho_t||_{L^p(m)}^p  = \sF(\mu_t) \le &~ (1-t)\iint_{X\times X}\frac{\beta_{1-t}(x_0,x_1)}{\rho_0(x_0)}F\left(\frac{\rho_0(x_0)}{\beta_{1-t}(x_0,x_1)}\right)d\sigma(x_0,x_1)\\
                &~ + t\iint_{X\times X}\frac{\beta_{t}(x_0,x_1)}{\rho_1(x_1)}F\left(\frac{\rho_1(x_1)}{\beta_{t}(x_0,x_1)}\right)d\sigma(x_0,x_1)\\
 = & ~ (1-t)\iint_{X\times X}\left(\frac{\rho_0(x_0)}{\beta_{1-t}(x_0,x_1)}\right)^{p-1}d\sigma(x_0,x_1)\\
                &~ + t\iint_{X\times X}\left(\frac{\rho_1(x_1)}{\beta_{t}(x_0,x_1)}\right)^{p-1}d\sigma(x_0,x_1)\\
 \le & ~ (1-t)\iint_{X\times X}\left(\frac{c}{L}\right)^{p-1}d\sigma(x_0,x_1)
                 + t\iint_{X\times X}\left(\frac{c}{L}\right)^{p-1}d\sigma(x_0,x_1) \\
 = &~ \left(\frac{c}L\right)^{p-1},
\end{align*}
 where $\sigma = (e_0,e_1)_\#\pi$ and $L = \inf\{\beta_{t}(x_0,x_1) :  t\in[0,1], d(x_0,x_1) \le D\}$. We then have
 \[
  ||\rho_t||_{L^\infty(m)} \le \lim_{p \to \infty} \left(\frac{c}L\right)^{\frac{p-1}p} = \frac{c}L.
 \]

 So it remains to estimate $L$ from below for different $N$ and $K$. If $K = 0$ or $N = 1$ we have $L = 1$.
 If $N = \infty$ we have
 \[
  \beta_{t}(x_0,x_1) = e^{\frac{1}{6}K(1-t^2)d(x_0,x_1)^2} \ge e^{\frac{1}{6}KD^2}.
 \]
 Finally, if $1 < N < \infty$ and $K < 0$ we get
 \begin{align*}
   \beta_{t}(x_0,x_1) & = \left(\frac{\sinh(t\alpha)}{t\sinh \alpha}\right)^{N-1} \ge \lim_{s \downarrow 0} \left(\frac{\sinh(s\alpha)}{s\sinh \alpha}\right)^{N-1}
                       = \left(\frac{2\alpha}{e^\alpha - e^{-\alpha}}\right)^{N-1} \\
   & \ge e^{-\alpha(N-1)} \ge \exp\left(-\sqrt{\frac{|K|}{N-1}}D(N-1)\right) = e^{-\sqrt{(N-1)|K|}D}.
 \end{align*} 
\end{proof}

Next we deduce the sharp bound from the strong displacement convexity of $\sE_\infty$.

\begin{lemma}\label{lma:EVI}
 Suppose that $\sE_\infty$ (or $\sE_N$) is strongly displacement convex in a geodesic metric measure space $(X,d,m)$.
 Let $\mu, \nu \in \mathcal{P}^{ac}(X,m)$ be measures with densities bounded from above by a constant $c$.
 Then for every $\pi \in \GeoOpt(\mu,\nu)$ we have for every $t \in [0,1]$
 \[
   \rho_t(x) \le c
 \]
 at $m$-almost every $x \in X$, where $\rho_t$ is the density of $(e_t)_\#\pi$ with respect to $m$.
\end{lemma}
\begin{proof}
 We prove the claim for $\sE_\infty$. The proof for $\sE_N$ is essentially the same.
 Take any
 \[
  \pi \in \GeoOpt\left(\mu,\nu\right)
 \]
 and let $\rho_t$ be the density of $(e_t)_\#\pi$ with respect to $m$. 
 Take $t \in (0,1)$ and $a \in \R$ so that $m(A_a) > 0$ with $A_a = \{x \in X : \rho_t(x) \ge a\}$.
 Define $\Gamma \subset \Geo(X)$ as
 \[
  \Gamma = \{\gamma \in \Geo(X) : \gamma(t) \in A_a\}.
 \]

 Because of the strong displacement convexity assumption $\sE_\infty$ is convex also along $\frac1{\pi(\Gamma)}\pi|_\Gamma$.
 Write $\tilde\rho_t$ for the density of $(e_t)_\#\frac1{\pi(\Gamma)}\pi|_\Gamma$ with respect to $m$.
 From the displacement convexity of $\sE_\infty$ along the geodesic $\frac1{\pi(\Gamma)}\pi|_\Gamma$ we get an upper bound
 \begin{align*}
  \int_{A_a}\frac{\rho_t}{\pi(\Gamma)}\log \frac{\rho_t}{\pi(\Gamma)}dm & = \int_{X}\tilde\rho_t\log \tilde\rho_tdm 
        \le (1-t)\int_{X}\tilde\rho_0\log \tilde\rho_0dm  + t\int_{X}\tilde\rho_1\log \tilde\rho_1dm \\
      & \le (1-t)\int_{X}\tilde\rho_0\log \frac{c}{\pi(\Gamma)}dm  + t\int_{X}\tilde\rho_1\log\frac{c}{\pi(\Gamma)}dm
       = \log\frac{c}{\pi(\Gamma)}.
 \end{align*}
 On the other hand, from Jensen's inequality we obtain the lower bound
 \begin{align*}
 \int_{A_a}\frac{\rho_t}{\pi(\Gamma)}\log \frac{\rho_t}{\pi(\Gamma)}dm & = m(A_a) \left(\dashint_{A_a}\frac{\rho_t}{\pi(\Gamma)}\log \frac{\rho_t}{\pi(\Gamma)}dm \right)\\
     & \ge m(A_a) \left(\dashint_{A_a}\frac{\rho_t}{\pi(\Gamma)}dm \right)\log \left(\dashint_{A_a}\frac{\rho_t}{\pi(\Gamma)}dm \right) \\
     & = \log \left(\dashint_{A_a}\frac{\rho_t}{\pi(\Gamma)}dm \right) 
      \ge \log \left(\dashint_{A_a}\frac{a}{\pi(\Gamma)}dm \right) =\log\frac{a}{\pi(\Gamma)}.
 \end{align*}
 Combining these bounds we get $a \le c$ and thus we have proven the claim. 
\end{proof}

Now we are ready to prove Theorem \ref{thm:CD1} noting that the rest of the weak local Poincar\'e inequalities follow with a similar proof.

\begin{proof}[Proof of Theorem \ref{thm:CD1}]
 Suppose that $u$, $g$, $x$ and $r$ are given. Let us abbreviate $B = B(x,r)$. Define $M$ as the median of $u$ in the ball
 $B$. In other words,
 \[
   M = \inf\left\{a \in \R : m(\{u > a\}) \le \frac{m(B)}2\right\}.
 \]
 Using the median $M$ we split the ball $B$ into two Borel sets $B^+$ and $B^-$ so that $B = B^+ \cup B^-$,
 $B^+ \cap B^- = \emptyset$, $m(B^+) = m(B^-)$ and
 \begin{equation}\label{eq:separation}
   u(x) \le M \le u(y) \qquad \text{for all } (x,y) \in B^- \times B^+.
 \end{equation} 
 These sets exist because the measure $m$ has no atoms and thus we can split the level set $u^{-1}(M)\cap B$, if necessary,
 so as to make the sets $B^+$ and $B^-$ have the same measure. Let
 \[
  \pi \in \GeoOpt\left(\frac1{m(B^+)}m|_{B^+},\frac1{m(B^-)}m|_{B^-}\right)
 \]
 be the geodesic given by Lemma \ref{lma:CD} and let $\rho_t$ be the density of $(e_t)_\#\pi$ with respect to $m$. 
 By Lemma \ref{lma:CD} we have for all $t \in [0,1]$ at $m$-almost every $y\in X$
 \[
  \rho_t(y) \le \frac{2}{m(B)}e^{\frac16|K|(2r)^2}\le \frac{2}{m(B)}e^{|K|r^2}.
 \]

 Now observe that from \eqref{eq:separation} we get an equality
 \[
   |u(\gamma(0)) - u(\gamma(1))| = |u(\gamma(0)) - M| + |M - u(\gamma(1))|
 \]
 for $\pi$-almost every $\gamma \in \Geo(X)$. Therefore
 \begin{align*}
  \int_{\Geo(X)} & |u(\gamma(0)) -  u(\gamma(1))| d\pi(\gamma)\\
   & = \int_{\Geo(X)} |u(\gamma(0)) - M|d\pi(\gamma) + \int_{\Geo(X)}|M - u(\gamma(1))|d\pi(\gamma)\\
   & = \frac{2}{m(B)}\int_{B^+}|u(x) - M|dm(x) + \frac{2}{m(B)}\int_{B^-}|M-u(x)|dm(x)\\
   & = \frac{2}{m(B)}\int_{B}|u(x) - M|dm(x).
 \end{align*}
 The rest of the proof follows the same lines as the proof of \cite[Theorem 2.5]{LV2007}.
 Let us repeat the key steps here for the convenience of the reader.
 Notice that $\pi$-almost every $\gamma \in \Geo(X)$ is contained in the ball $B(x,2r)$. Thus
 \begin{align*}
  \int_{B(x,r)}|u - \langle u\rangle_{B(x,r)}|dm & \le \frac{1}{m(B)}\iint_{B \times B} |u(x) - u(y)|dm(x)dm(y) \\
   & \le \frac{1}{m(B)}\iint_{B \times B} (|u(x) - M| + |M - u(y)|)dm(x)dm(y) \\
   & = 2\int_{B} |u(x) - M|dm(x) \\
   & = m(B)\int_{\Geo(X)} |u(\gamma(0)) - u(\gamma(1))| d\pi(\gamma)\\
   & \le 2rm(B) \int_{\Geo(X)}\int_0^1 g(\gamma(t))dt d\pi(\gamma)\\
   & = 2rm(B) \int_0^1 \int_Xg(x)\rho_t(x)dm(x)dt\\
   & \le 4re^{|K|r^2} \int_0^1 \int_{B(x,2 r)}g(x)dm(x)dt = 4 re^{|K|r^2} \int_{B(x,2 r)}gdm.
 \end{align*}
\end{proof}

\begin{remarks}
 (i) In the proof of Theorem \ref{thm:CD1} above we obtained only a weak version of the local Poincar\'e inequality because 
 in general it might be that a geodesic between two points $y,z \in B(x,r)$ does not stay inside the ball $B(x,r)$.
 In a metric space  $(X,d)$ where balls are convex in the sense that all the geodesics between any two points in a ball
 stay inside the ball the above proof immediately gives a strong local Poincar\'e inequality. For example in those
 $CD(0,\infty)$ spaces where the balls are convex we have
 \[
  \dashint_{B(x,r)}|u - \langle u\rangle_{B(x,r)}|dm \le 4 r \dashint_{B(x,r)}gdm.
 \]

 (ii) In a doubling geodesic metric space a weak local Poincar\'e inequality implies a strong local Poincar\'e inequality.
 When we move from the weak inequality to the strong inequality the constant usually has to be enlarged. However,
 from the assumption that the entropy $\sE_N$ is strongly displacement convex we can directly prove
 a strong local Poincar\'e inequality with the same constant as in Corollary \ref{cor:EVI2}:
 \begin{equation}\label{eq:strong}
   \dashint_{B(x,r)}|u - \langle u\rangle_{B(x,r)}|dm \le 2^{N+2} r \dashint_{B(x,r)}gdm.
 \end{equation}
 With two simple observations we can modify the proof of Theorem \ref{thm:CD1} to give this inequality.
 The first observation is that we do not have to use geodesics in the proof but instead we can take a collection of curves whose
 lengths are uniformly bounded from above. As we will shortly see, suitable curves can be constructed by combining geodesics.
 The second observation is that a geodesic from a point $y \in B(x,r)$ to the center $x$ always stays inside the ball $B(x,r)$. 

 With a similar
 proof as for Lemma \ref{lma:EVI} we can show that for a geodesic $\pi^1$ from $\frac1{m(B^+)}m|_{B^+}$ to $\delta_x$
 we have for all $t \in [0,\frac12]$ at $m$-almost $y\in X$ the estimate
 \begin{equation}\label{eq:curvebound}
  \rho_t^1(y) \le \frac{2^{N+1}}{m(B)},
 \end{equation}
 where $\rho_t^1$ is the density of $(e_t)_\#\pi$ with respect to $m$. 
 The same argument works for a geodesic $\pi^2$ from $\frac1{m(B^-)}m|_{B^-}$ to $\delta_x$.

 By Lemma \ref{lma:EVI} we know that for any $\pi^3 \in \GeoOpt((e_\frac12)_\#\pi^1,(e_\frac12)_\#\pi^2)$ we have for 
 every $t \in [0,1]$ at $m$-almost every $y \in X$ the same bound \eqref{eq:curvebound}
 for the density of $(e_t)_\#\pi^3$ with respect to $m$. If we combine these three geodesics to a rectifiable curve
 $\Gamma \colon [0,1] \to \mathcal{P}(X)$ by defining
 \[
  \Gamma(t) = \begin{cases}
               (e_{2t})_\#\pi^1 & \text{if } 0 \le t \le \frac14,\\
               (e_{2(t-\frac14)})_\#\pi^3 & \text{if } \frac14 < t < \frac34,\\
               (e_{2(1-t)})_\#\pi^2 & \text{if } \frac34 \le t \le 1,\\
              \end{cases}
 \]
 we get a curve in $\mathcal{P}(X)$ joining $\frac1{m(B^+)}m|_{B^+}$ to $\frac1{m(B^-)}m|_{B^-}$ so that the support of $\Gamma(t)$
 is always inside $B$ and we have for all $t\in[0,1]$ the upper bound $\frac{2^{N+1}}{m(B)}$ for the density of $\Gamma(t)$ at
 $m$-almost all $y \in X$. Also, when we consider the curve $\Gamma$ as a measure on the rectifiable curves of $X$, this measure
 is concentrated on curves in $X$ that have length bounded above by $2r$. Therefore, the same estimates as in the proof of
 Theorem \ref{thm:CD1} give us the strong local Poincar\'e inequality \eqref{eq:strong}.
\end{remarks}

\section{Proof for the almost uniqueness of geodesics}\label{sec:uniqueness}

Let us now turn to the proof of Theorem \ref{thm:uniqueness}. We explain the idea behind the proof with the help of Figure \ref{fig:idea}.
By assumption the entropy functional $\sE_N$ is convex along the geodesic in the space of measures which connects
measure $\mu_0$ to $\mu_1 = \delta_x$ using only geodesics in the upper
part of the space. The graph of $\sE_N$ along this geodesic is drawn on the right with a dashed line. Now consider a new
geodesic between $\mu_0$ and $\mu_1$ that moves half of the measure using the upper part and half of the measure using the lower
part of the space. With this geodesic the support of the transported measure is split to the upper and lower parts of the space.
At the time when the measure is fully split the entropy is $2^{1/N}$ times the entropy of the corresponding measure along the
geodesic which uses only the upper part of the space. On the graph in the Figure \ref{fig:idea} the entropy along this new
geodesic is drawn with a solid line. On the time intervals $[t_1,t_2]$ and $[t_3,t_4]$, when the support of the measure travels past the branching points
of the space, the entropy has a dramatic change. This change contradicts the convexity of the entropy functional along the new
geodesic proving that such branching space does not satisfy the strong convexity assumption.

\begin{figure}
   \psfrag{m0}{$\mu_0$}
   \psfrag{m1}{$\mu_1$}
   \psfrag{t}{$t$}
   \psfrag{1}{$1$}
   \psfrag{0}{$0$}
   \psfrag{t1}{$t_1$}
   \psfrag{t2}{$t_2$}
   \psfrag{t3}{$t_3$}
   \psfrag{t4}{$t_4$}
   \psfrag{E}{$\sE_N(\mu_t)$}
   \centering
   \includegraphics[width=\textwidth]{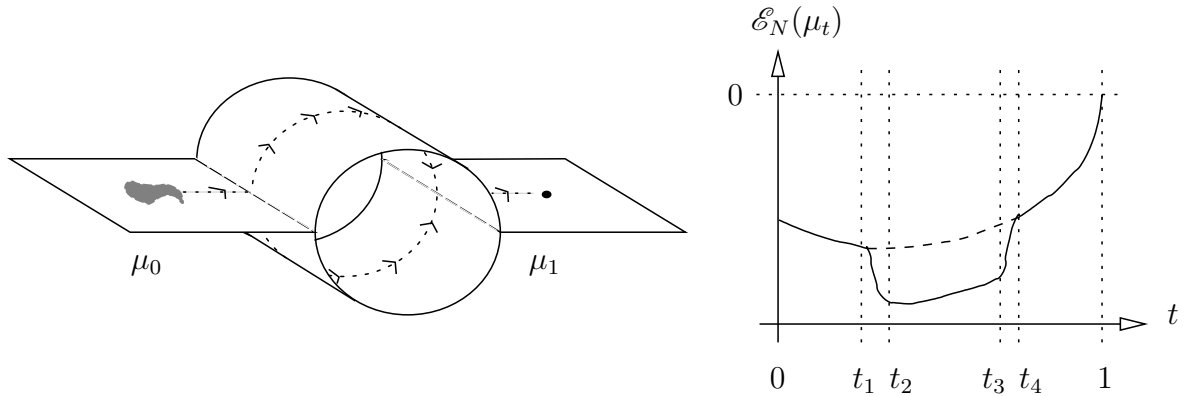}
   \caption{An illustration of the proof of Theorem \ref{thm:uniqueness}.}
   \label{fig:idea}
\end{figure}

Guided by the idea presented above we try to reduce the general case to a situation corresponding to that in Figure \ref{fig:idea}.
The first step is to show that if there are multiple geodesics joining many points in the space then there are also geodesics
which agree on some initial time interval and then branch out. Next we choose a subset from these geodesics for which the branching
happens roughly at the same time and so that all of these geodesics branch out sufficiently. The first requirement guarantees
that the large change in the entropy happens on a small enough time interval, as in Figure \ref{fig:idea}. The second requirement
tells us that the supports of the two branches of the measure are really disjoint, justifying the calculations for the drop in the
entropy. Let us now make these steps rigorous.

\begin{proof}[Proof of Theorem \ref{thm:uniqueness}]
 Suppose that the claim is not true. Let $x \in X$ be a point so that the set
 \[
  A = \{y \in X : \text{there exist two distinct geodesics between }x \text{ and }y\}
 \]
 has positive $m$-measure.

 We already know that the measure $m$ is doubling so in particular it is outer regular.
 Let us use this to show that there are lots of points $y \in A$ where the two
 distinct geodesics from $y$ to $x$ agree on a small initial time interval.

 Let $z \in X$ and $r > 0$ be so that
 \[
  m(B(z,r)\cap A) \ge \frac23 m(B(z,r))  
 \]
 and $m(S(z,r)) = 0$. (Actually, we know from \cite[Theorem 2.3]{S2006II} that either $m$ is a
 point mass or $m(S(z,r)) = 0$ for all $z \in X$ and $r > 0$.) Take a small $0 < s < 1$
 and
 \[
  \pi\in \GeoOpt\left(m(A \cap B(z,r))^{-1}m|_{A\cap B(z,r)},\delta_x\right).
 \]
 From the convexity of $\sE_N$ we get
 \[
  \int_X\rho_s^{1-\frac{1}{N}}dm \ge (1-s)\int_X\rho_0^{1-\frac{1}{N}}dm = (1-s)m(A \cap B(z,r))^{\frac{1}{N}},
 \]
 where $\rho_sm$ is the absolutely continuous part of the measure $(e_s)_\#\pi$ with respect to $m$.
 On the other hand, we always have by Jensen's inequality
 \begin{align*}
  \int_X\rho_s^{1-\frac{1}{N}}dm & = m(\{\rho_s>0\})\left(\frac{1}{m(\{\rho_s>0\})}\int_{\{\rho_s>0\}}\rho_s^{1-\frac{1}{N}}dm\right) \\
   & \le m(\{\rho_s>0\})\left(\frac{1}{m(\{\rho_s>0\})}\int_{\{\rho_s>0\}}\rho_sdm\right)^{1-\frac{1}{N}} \le m(\{\rho_s>0\})^{\frac{1}{N}}.
 \end{align*}
 Combining the previous two estimates we see that
 \[
   m(\{\rho_s>0\}) \ge(1-s)^Nm(A \cap B(z,r)) \to m(A \cap B(z,r))
 \]
 as $s \downarrow 0$. We also have $\{\rho_s>0\} \subset B(z,r_s)$ with $r_s \downarrow r$ as $s \downarrow 0$.
 Recalling that $m(S(z,r)) = 0$ we get
 \[
   m(\{\rho_s>0\}\cap B(z,r)) \to m(A \cap B(z,r))
 \]
 as $s \downarrow 0$, implying that
 \[
  m(A \cap \{\rho_s>0\} \cap B(z,r)) > 0  
 \]
 for sufficiently small $s$.
 Therefore also the set
 \begin{align*}
  A_1 = \{y \in B(z,r_\epsilon) : ~& \text{there exist two distinct geodesics between }x \text{ and }y\\
                   & \text{which agree close to } y\} 
 \end{align*}
 has positive $m$-measure.
 
 Now that we have established that there is enough branching away from the initial time let us make the next reduction
 to geodesics which branch out roughly at the same time.

 The open interval $(0,1)$ can be covered by a countable number of intervals of the form $[T_1,T_2]$ with the restriction
 \begin{equation}\label{eq:smallinterval}
  \max\left\{\frac{T_2}{T_1}, \frac{1-T_1}{1-T_2}\right\} \le 2^{\frac{1}{2N}}.
 \end{equation}
 Therefore there exist some such $T_1$ and $T_2$ for which the set
 \begin{align*}
  A_2 = \{y \in A_1 : ~& \text{there exist two distinct geodesics from }y \text{ to }x \text{ which}\\
                   & \text{agree on the interval }[0,T_1]\text{ but not on }[0,T_2]\} 
 \end{align*}
 has positive $m$-measure. Now let $\delta > 0$ so that the set
 \begin{align*}
  A_3 = \{y \in A_2 : ~& \text{there exist two distinct geodesics } \gamma, \gamma'\text{ from }y \text{ to }x \text{ which agree}\\
                   & \text{on the interval }[0,T_1]\text{ and } d(\gamma(t),\gamma'(t)) > 2\delta \text{ for some }t \in [T_1,T_2]\} 
 \end{align*}
 has positive $m$-measure.

 Again by going into a subset we find a time $t_2 \in [T_1,T_2]$ so that the set
 \begin{align*}
  A_4 = \{y \in A_3 : ~& \text{there exist two distinct geodesics } \gamma, \gamma'\text{ from }y \text{ to }x \text{ which}\\
                   & \text{agree on the interval }[0,T_1]\text{ and } d(\gamma(t_2),\gamma'(t_2)) > \delta\} 
 \end{align*}
 has positive $m$-measure.

 Let us write $t_1 = T_1$. Notice that $t_1$ and $t_2$ also satisfy the estimate \eqref{eq:smallinterval} which the original $T_1$
 and $T_2$ did.

 Because $m(A_4)>0$, there is some $w \in X$ so that the set
 \begin{align*}
  E = \{y \in A_4 : ~&\text{there exists }\gamma \in \Geo(X) \text{ with }\gamma(0) = y, \gamma(1) = x \\
                    &\text{ and } \gamma(t_2) \in B(w, \delta/2)\}
 \end{align*}
 has positive $m$-measure.
 Let us then select for all $y \in E$ a pair of geodesics $\gamma_y$, $\tilde\gamma_y$ so that $\gamma_y(0) = \tilde\gamma_y(0) = y$, 
 $\gamma_y(1) = \tilde\gamma_y(1) = x$, $\gamma_y(t) = \tilde\gamma_y(t)$ for all $t \in [0,t_1]$,
 $\gamma_y(t_2) \in B(w, \delta/2)$ and $\tilde\gamma_y(t_2) \notin B(w, \delta/2)$.
 Using these pairs we write $G_1 = \{\gamma_y : y \in E\}$ and $G_2 = \{\tilde\gamma_y : y \in E\}$.

 Next we define two measures $\pi_1, \pi_2 \in \GeoOpt((m(E))^{-1}m|_E, \delta_x)$. The first one is defined as
 \[
  \pi_1(F) = \frac{m(\{y \in E : \text{exists }\gamma \in F\cap G_1 \text{ such that } \gamma(0)= y\})}{m(E)} 
 \]
 and the second one as
 \[
  \pi_2(F) = \frac{m(\{y \in E : \text{exists }\gamma \in F\cap G_2 \text{ such that } \gamma(0)= y\})}{m(E)}.
 \]

 The definition of $G_1$ and $G_2$ guarantee that $m(\{\rho_{1,t_2}>0\} \cap \{\rho_{2,t_2}>0\}) = 0$.
 Here $\rho_{1,s}m$ and $\rho_{2,s}m$ are the absolutely continuous parts of $(e_s)_\#\pi_1$ and $(e_s)_\#\pi_2$
 with respect to the measure $m$. Using the convexity of $\sE_N$
 first to the geodesic $(\pi_1+\pi_2)/2$ between times $0$ and $t_2$ and then separately to the geodesics $\pi_1$ and $\pi_2$
 between times $t_1$ and $1$, and finally using the inequality \eqref{eq:smallinterval} we arrive at the contradiction
 \begin{align*}
  \int_X(\rho_{1,t_1})^{1-\frac{1}{N}}dm & \ge \frac{t_2-t_1}{t_2}m(E)^{\frac{1}{N}} 
   + \frac{t_1}{t_2}\left(\int_X\left(\frac{\rho_{1,t_2}}{2}\right)^{1-\frac{1}{N}}dm +\int_X\left(\frac{\rho_{2,t_2}}{2}\right)^{1-\frac{1}{N}}dm\right)\\
  & > \frac{t_1}{t_2}2^{\frac{1}{N}-1}\left(\int_X\left(\rho_{1,t_2}\right)^{1-\frac{1}{N}}dm +\int_X\left(\rho_{2,t_2}\right)^{1-\frac{1}{N}}dm\right)\\
  & \ge \frac{t_1}{t_2}2^{\frac{1}{N}-1} \frac{1-t_2}{1-t_1}\left(\int_X(\rho_{1,t_1})^{1-\frac{1}{N}}dm + \int_X(\rho_{2,t_1})^{1-\frac{1}{N}}dm\right)\\
  & = \frac{t_1}{t_2}2^{\frac{1}{N}} \frac{1-t_2}{1-t_1}\int_X(\rho_{1,t_1})^{1-\frac{1}{N}}dm \ge \int_X(\rho_{1,t_1})^{1-\frac{1}{N}}dm
 \end{align*}
 thus proving the claim.
\end{proof}

\end{document}